\documentclass[11pt]{amsart}
\usepackage{color,amssymb}

\usepackage[all]{xy}

\newtheorem{theorem}{Theorem}[section]
\newtheorem{claim}[theorem]{Claim}
\newtheorem{proposition}[theorem]{Proposition}
\newtheorem{corollary}[theorem]{Corollary}
\newtheorem{lemma}[theorem]{Lemma}

\newtheorem{question}[theorem]{Question}

\theoremstyle{definition}
\newtheorem{remark}[theorem]{Remark}
\newtheorem{definition}[theorem]{Definition}
\newtheorem{example}[theorem]{Example}

\newcommand{\IN}{\mathbb N}

\newcommand{\IQ}{\mathbb Q}
\newcommand{\IC}{\mathbb C}
\newcommand{\IT}{\mathbb T}
\newcommand{\IZ}{\mathbb Z}
\newcommand{\F}{\mathcal F}
\newcommand{\U}{\mathcal U}
\newcommand{\C}{\mathcal C}
\newcommand{\V}{\mathcal V}

\newcommand{\w}{\omega}

\newcommand{\Fil}{\mathrm{\varphi}}
\newcommand{\Ra}{\Rightarrow}
\newcommand{\Haus}{\mathsf{T_{\!2}S}}

\newcommand{\Zero}{\mathsf{T_{\!z}S}}

\newcommand{\Tau}{\mathcal T}

\newcommand{\korin}[2]{\!\sqrt[#1]{\!#2}}

\title{Characterizing categorically closed commutative semigroups}

\author{Taras Banakh, Serhii Bardyla}
\address{T.Banakh: Ivan Franko National University of Lviv (Ukraine) and Jan Kochanowski University in Kielce (Poland)}
\email{t.o.banakh@gmail.com}
\address{S.~Bardyla: University of Vienna, Institute of Mathematics, Kurt G\"{o}del Research Center (Austria)}
\thanks{The second author was supported by the Austrian Science Fund FWF (Grant  M 2967).}
\email{sbardyla@yahoo.com}
\makeatletter
\@namedef{subjclassname@2020}{%
  \textup{2020} Mathematics Subject Classification}
\makeatother


\begin{document}
\begin{abstract} Let $\C$ be a class of $T_1$ topological semigroups which contains all zero-dimensional Hausdorff topological semigroups. A semigroup $X$ is called {\em $\C$-closed} if
$X$ is closed in each topological semigroup $Y\in \C$ containing $X$ as a discrete subsemigroup;  $X$ is {\em projectively $\C$-closed} if for each congruence $\approx$ on $X$ the quotient semigroup $X/_\approx$  is $\C$-closed. A semigroup $X$ is called {\em chain-finite} if for any infinite set $I\subseteq X$ there are elements $x,y\in I$ such that $xy\notin\{x,y\}$. We prove that a semigroup $X$ is $\C$-closed if it admits a homomorphism $h:X\to E$ to a chain-finite semilattice $E$ such that for every $e\in E$ the semigroup $h^{-1}(e)$ is $\C$-closed. Applying this theorem, we prove that a commutative semigroup $X$ is $\C$-closed if and only if $X$ is periodic, chain-finite, all subgroups of $X$ are bounded,  and for any infinite set $A\subseteq X$ the product $AA$ is not a singleton. A commutative semigroup $X$ is projectively $\C$-closed if and only if $X$ is chain-finite, all subgroups of $X$ are bounded and the union $H(X)$ of all subgroups in $X$ has finite complement $X\setminus H(X)$.
\end{abstract}
\maketitle

\section{Introduction and Main Results}

In many cases,  completeness properties of various objects of General Topology or  Topological Algebra can be characterized externally as closedness in ambient objects. For example, a metric space $X$ is complete if and only if $X$ is closed in any metric space containing $X$ as a subspace. A uniform space $X$ is complete if and only if $X$ is closed in any uniform space containing $X$ as a uniform subspace. A topological group $G$ is Ra\u\i kov complete  if and only if it is closed in any topological group containing $G$ as a subgroup.

On the other hand, for topological semigroups there are no reasonable notions of (inner) completeness. Nonetheless we can define many completeness properties of  semigroups via their closedness in ambient topological semigroups. 

A {\em topological semigroup} is a topological space $X$ endowed
with a continuous associative binary operation $X\times X\to
X$, $(x,y)\mapsto xy$.

\begin{definition}
Let $\C$ be a class of topological semigroups.

A topological semigroup $X$ is called  {\em $\C$-closed\/} if for any isomorphic topological
embedding $h:X\to Y$ to a topological semigroup $Y\in\C$ the image $h[X]$ is closed in $Y$.

A semigroup $X$ is called {\em $\C$-closed} if so is the topological semigroup $X$ endowed with the discrete topology.
\end{definition}

$\C$-closed topological groups for various classes $\C$ were investigated by many authors including  Arhangel’skii, Banakh, Choban, Dikranjan, Goto, Luka\'sc and Uspenskij~\cite{AC,Ban,DU,G,L,U}. Closedness of commutative topological groups in the class of Hausdorff topological semigroups was investigated by Keyantuo and Y. Zelenyuk~\cite{Z3,Z1}. Semigroup compactifications of locally compact commutative groups were investigated by Zelenyuk~\cite{Z2,Z4};
$\mathcal{C}$-closed topological semilattices were investigated by Gutik, Repov\v{s}, Stepp and the authors in~\cite{BBm, BB, BBs, BBw, BBG, BBR, GutikPagonRepovs2010, GutikRepovs2008,Stepp69,Stepp75}.
For more information about complete topological semilattices and pospaces we refer to the recent survey of the authors~\cite{BBc}.


We shall be interested in the $\C$-closedness
for the classes:
\begin{itemize}
\item $\Haus$ of Hausdorff topological semigroups;
\item $\Zero$ of zero-dimensional Hausdorff topological
semigroups;
\item $\mathsf{T_{\!1}S}$ of topological semigroups satisfying the separation axiom $T_1$.
\end{itemize}

Recall that a topological space is {\em zero-dimensional} if it has a base of
the topology consisting of {\em clopen} (=~closed-and-open)
subsets. A topological space $X$ satisfies the separation axiom $T_1$ if each finite subset is closed in $X$.

Since $\Zero\subseteq\Haus\subseteq\mathsf{T_{\!1}S}$, for every semigroup we have the implications:
$$\mbox{$\mathsf{T_{\!1}S}$-closed $\Ra$ $\mathsf{T_{\!2}S}$-closed $\Ra$ $\mathsf{T_{\!z}S}$-closed}.$$
From now on we assume that $\C$ is a class of topological semigroups such that
$$\Zero\subseteq\C\subseteq\mathsf{T_{\!1}S}.$$

Now we recall two known characterizations of $\C$-closedness for semilattices and groups.

Let us mention that a semigroup $X$ is called a {\em semilattice} if it is commutative and each element $x\in X$ is an {\em idempotent}, which means that $xx=x$. Each semilattice $E$ carries a partial order $\le$ defined by $x\le y$ if and only if $xy=x$. In this case $xy=\inf\{x,y\}$. Given two elements $x,y$ of a semilattice  we write $x<y$ if $x\le y$ and $x\ne y$.

A subset $C$ of a semigroup $X$ is called a {\em chain} if $xy\in\{x,y\}$ for any elements $x,y\in C$. A subset $C$ of a semilattice is a chain if and only if any elements $x,y\in C$ are comparable in the partial order $\le$.

A semigroup $X$ is called {\em chain-finite} if $X$ contains no infinite chains. Observe that a commutative semigroup $X$ is chain-finite if and only if its maximal semilattice $E(X)=\{x\in X:xx=x\}$ is chain-finite. Lemma 2.6 from~\cite{BBc} implies that a chain-finite semilattice $S$ is down-complete with respect to the natural partial order $\leq$, that is  every nonempty subset $A\subseteq S$ has the greatest lower bound $\inf A\in S$. In particular, every chain-finite nonempty semilattice contains the least element. Also every chain-finite semilattice contains a maximal element (which is not necessary the greatest element of the semilattice).

The following characterization of $\C$-closed semilattices was proved in \cite{BBm}.

\begin{theorem}\label{BB} A semilattice $X$ is $\C$-closed if and only if $X$ is chain-finite.
\end{theorem}

A semigroup $X$ is defined to be
\begin{itemize}
\item {\em periodic} if for every $x\in X$ there exists $n\in\IN$ such that the power $x^n$ is an idempotent;
\item {\em bounded} if there exists
$n\in\IN$ such that for every $x\in X$ the power $x^n$ is an
idempotent.
\end{itemize}
It is clear that each bounded semigroup is periodic (but not vice versa).

The following characterization of $\C$-closed commutative groups was obtained by the first author in \cite{Ban}.

\begin{theorem}\label{B}  A commutative group $X$ is $\C$-closed if and only if $X$ is bounded.
\end{theorem}

In this paper we characterize $\C$-closed commutative semigroups.

Our principal tool for establishing the $\C$-closedness  is the following theorem.

\begin{theorem}\label{cool} A semigroup $X$ is $\C$-closed if $X$ admits a homomorphism $h:X\to E$ to a chain-finite semilattice $E$ such that for every $e\in E$ the semigroup $h^{-1}(e)$ is $\C$-closed.
\end{theorem}







Theorem~\ref{cool} will be applied in the proof of the following theorem, which is one of the two main results of this paper.

\begin{theorem}\label{t:main}  A commutative semigroup $X$ is $\C$-closed if and only if $X$ is periodic, chain-finite, all subgroups of $X$ are bounded, and for every infinite subset $A\subseteq X$ the set $AA=\{xy:x,y\in A\}$ is not a singleton.
\end{theorem}


\begin{corollary}\label{c:main2}
Each subsemigroup of a $\C$-closed commutative semigroup is $\C$-closed.
\end{corollary}

\begin{remark} 
By \cite[Proposition 10]{Ban}, the semidirect product $\IZ\rtimes \{-1,1\}$ (endowed with the binary operation $\langle x,i\rangle*\langle y,j\rangle=\langle x+i\cdot y,i\cdot j\rangle$) is a $\C$-closed group which is not bounded. This example shows that Theorem~\ref{B} and Corollary~\ref{c:main2} can not be generalized to non-commutative semigroups.
\end{remark}

\begin{example}\label{ex:Taimanov} Take any infinite cardinal $\kappa$ and endow it with the binary operation $*$ defined by
$$x*y=\begin{cases}
1&\mbox{if $x\ne y$ and $x,y\in\kappa\setminus\{0,1\}$};\\
0&\mbox{otherwise}.
\end{cases}$$
The semigroup $X=(\kappa,*)$ was introduced by Taimanov \cite{Taimanov}. Gutik \cite{Gutik} proved that the semigroup $X$ is $\mathsf{T_{\!1}S}$-closed but the quotient semigroup $X/I$ by the ideal $I=\{0,1\}$ is not $\mathsf{T_{\!z}S}$-closed.
\end{example}

Example~\ref{ex:Taimanov} shows that the $\C$-closedness is not preserved by taking quotient semigroups. This observation motivates introducing the definitions of ideally and projectively $\C$-closed semigroups.

Let us recall that a {\em congruence} on a semigroup $X$ is an equivalence relation $\approx$ on $X$ such that for any elements $x\approx y$ of $X$ and any $a\in X$ we have $ax\approx ay$ and $xa\approx ya$. For any congruence $\approx$ on a semigroup $X$, the quotient set $X/_\approx$ has a unique semigroup structure such that the quotient map $X\to X/_\approx$ is a semigroup homomorphism. The semigroup $X/_\approx$ is called the {\em quotient semigroup} $X$ by the congruence $\approx$.

A subset $I$ of a semigroup $X$ is called an {\em ideal} if $IX\cup XI\subseteq I$. Every ideal $I\subseteq X$ determines the congruence $(I\times I)\cup \{(x,y):x=y\}\subseteq X\times X$. The quotient semigroup of $X$ by this congruence is denoted by $X/I$ and called the {\em quotient semigroup} of $X$ by the ideal $I$. If $I=\emptyset$, then the quotient semigroup $X/\emptyset$ can be identified with the semigroup $X$.

A semigroup $X$ is called
\begin{itemize}
\item {\em projectively $\C$-closed} if for any congruence $\approx$ on $X$ the quotient semigroup $X/_{\approx}$ is $\C$-closed;
\item {\em ideally $\C$-closed} if for any ideal $I\subseteq X$ the quotient semigroup $X/I$ is $\C$-closed.
\end{itemize}
It is easy to see that for every semigroup the following implications hold:
$$\mbox{projectively $\C$-closed $\Ra$ ideally $\C$-closed $\Ra$ $\C$-closed.}$$

For a semigroup $X$ the union $H(X)$ of all subgroups of $X$ is called {\em the Clifford part} of $X$.

A semigroup $X$ is called
\begin{itemize}
\item {\em Clifford}  if $X=H(X)$;
\item {\em almost Clifford} if $X\setminus H(X)$ is finite.
\end{itemize}

The second main result of this paper is the following characterization.

\begin{theorem}\label{t:mainP} For a commutative semigroup $X$ the following conditions are equivalent:
\begin{enumerate}
\item $X$ is projectively $\C$-closed;
\item $X$ is ideally $\C$-closed;
\item the semigroup $X$ is chain-finite, almost Clifford, and all subgroups of $X$ are bounded.
\end{enumerate}
\end{theorem}

We do not know whether the equivalence $(1)\Leftrightarrow(2)$ in Theorem~\ref{t:mainP} remains true for any (semi)group, see Question~\ref{id}.

\begin{remark} Theorem~\ref{t:mainP} implies that the projective $\C$-closedness of commutative semigroups is inherited by subsemigroups and quotient semigroups.
\end{remark}

Theorems~\ref{t:main} and \ref{t:mainP} will be proved in Sections~\ref{s:main}, \ref{s:mainP} after a  preliminary work made in Sections~\ref{s:ultra}--\ref{s:Zero} and \ref{s:quotient}.

\section{The topological semigroup of filters on a semigroup}\label{s:ultra}

In this section for every semigroup $X$ we define the topological
semigroup $\Fil(X)$ of filters on $X$, containing $X$ as a dense
discrete subsemigroup. This construction is our principal tool in
the proofs of non-$\C$-closedness of  semigroups.

We recall that a {\em filter} on a set $X$ is any family $\F$ of
nonempty subsets of $X$, which is closed under finite
intersections and taking supersets in $X$.  A filter $\F$ is
\begin{itemize}
\item {\em free} if $\bigcap\F=\emptyset$;
\item {\em principal} if $\{x\}\in\F$ for some $x\in X$.
\end{itemize}

A subfamily $\mathcal
B\subseteq\F$ is called a {\em base} of a filter $\F$ if
$\F=\{A\subseteq X:\exists B\in\mathcal B\;\;(B\subseteq A)\}$.
 By $\Fil(X)$ we denote the set of all filters on $X$. The set
$\Fil(X)$ is partially ordered by the inclusion relation. Maximal
elements of the partially ordered set $\Fil(X)$ are called {\em
ultrafilters}. It is well-known that a filter $\F$ on $X$ is an {\em ultrafilter} if and only if
for any partition $X=U\cup V$ of $X$ either $U$ or $V$ belongs to
$\F$. By $\beta(X)\subseteq \Fil(X)$ we denote the set of all
ultrafilters on $X$.

 Each point $x\in X$ will be identified with the principal
ultrafilter $\U_x=\{U\subseteq X:x\in U\}\in\beta(X)\subseteq\Fil(X)$.
So, $X$ can be identified with the subset of $\beta(X)$ consisting
of all principal ultrafilters. Thus we get the chain of
inclusions $X\subseteq \beta(X)\subseteq\Fil(X)$.

The set $\Fil(X)$ carries the {\em canonical topology} generated
by
the base consisting of the sets $$\langle
U\rangle=\{\U\in\Fil(X):U\in\U\}$$ where $U\subseteq X$ runs over
subsets of $X$. It can be shown that this topology satisfies the
separation axiom $T_0$, i.e., for each distinct points $x,y\in \Fil(X)$ there exists an open set which contains precisely one of them. The set $X$ of principal ultrafilters is
dense in $\Fil(X)$ and for each $x\in X$ the singleton
$\{x\}=\{F\in\Fil(X):\{x\}\in F\}=\langle\{x\}\rangle$ is an open
set in $\Fil(X)$. So, $X$ is a dense discrete subspace of $\Fil(X)$.
The subspace $\beta(X)$ of ultrafilters is compact, Hausdorff, zero-dimensional, and
dense in $\Fil(X)$. Consequently, each subspace of $\beta(X)$ is
zero-dimensional and Tychonoff.

If $X$ is a (commutative) semigroup, then $\Fil(X)$ has a natural
structure of a (commutative) topological semigroup: for any
filters
$\U,\V\in\Fil(X)$ their product $\U\V$ is the filter generated by
the base $\{UV:U\in\U,\;V\in\V\}$, where $UV=\{uv:u\in U,\;v\in
V\}$. Every neighborhood of $\U\V$ in $\Fil(X)$ contains a basic
neighborhood $\langle UV\rangle$ for some $U\in\U$ and $V\in\V$.
Then $\langle U\rangle$ and $\langle V\rangle$ are basic
neighborhoods of the filters $\U,\V$ in $\Fil(X)$ such that
$\langle U\rangle\cdot\langle V\rangle\subseteq\langle UV\rangle$, which means that
$\Fil(X)$
is a topological semigroup, containing $X$ as a dense discrete
subsemigroup. Observe that the product of two ultrafilters is not
necessarily an ultrafilter, so $\beta(X)$ is not necessarily a
subsemigroup of $\Fil(X)$.

\section{Some properties of periodic semigroups}

In this section we establish some properties of periodic semigroups. Let us recall that a semigroup $S$ is {\em periodic} if for every $x\in S$ there exists $n\in\IN$ such that $x^n$ is an idempotent of $S$.

For a subset $A$ of a semigroup $S$, let $$\korin{\infty}{A}=\{x\in X:\exists n\in\IN\;\;(x^n\in A)\}.$$ For an element $e\in S$, the set $\korin{\infty}{\,\{e\}}$ will be denoted by $\korin{\infty}{\,e}$. Observe that a semigroup $S$ is periodic if and only if $S=\bigcup_{e\in E(S)}\korin{\infty}{\,e}$, where $E(S)=\{e\in S:ee=e\}$ is the set of idempotents of $S$.

For an element $a$ of a semigroup $S$ the set
$$H_a=\{x\in S:(xS^1=aS^1)\;\wedge\;(S^1x=S^1a)\}$$
is called the {\em $\mathcal H$-class} of $a$.
Here $S^1=S\cup\{1\}$ where $1$ is an element such that $1x=x=x1$ for all $x\in S^1$. We shall assume that $x^0=1$ for every $x\in S^1$.

By Corollary 2.2.6 \cite{Howie}, for every idempotent $e\in E(S)$ its $\mathcal H$-class $H_e$ coincides with the maximal subgroup of $S$, containing the idempotent $e$.
The union $$H(S)=\bigcup_{e\in E(S)}H_e$$is the {\em Clifford part} of $S$. 

For two subsets $A,B$ of a semigroup $S$ their product  in $S$ is defined as  $$A\cdot B=\{ab:a\in A,\;b\in B\}.$$The set $A\cdot B$ will be also denoted by $AB$.

\begin{lemma}\label{l:C-ideal} For any idempotent $e$ of a semigroup $S$ we have
$$
(\korin{\infty}{H_e}\cdot H_e)\cup (H_e\cdot \korin{\infty}{H_e})\subseteq H_e.$$
\end{lemma}

\begin{proof} Fix any element $x\in \korin{\infty}{H_e}$. First we prove that $xe\in H_e$. Since $x\in\korin{\infty}{H_e}$, there exists $n\in\IN$ such that $x^n\in H_e$ and hence $x^{2n}\in H_e$. Observe that $xeS^1=xx^{n}S^1\subseteq x^nS^1=eS^1$ and $eS^1=x^{2n}S^1\subseteq x^{n+1}S^1=xeS^1$ and hence $xeS^1=eS^1$. By analogy we can prove that $S^1xe=S^1e$. Then $xe\in H_e$ by the definition of the $\mathcal H$-class $H_e$.

Now fix any element $y\in H_e$. Since $H_e$ is a group with neutral element $e$, we obtain $xy=x(ey)=(xe)y\in H_eH_e=H_e$. By analogy we can prove that $yx\in H_e$.
\end{proof}

Let $S$ be a periodic semigorup. Then for any $x\in S$ there exists $n\in\mathbb N$ and idempotent $e\in S$ such that $e=x^n$. Assume that there exist $m\in\mathbb N$ and idempotent $f\in S$ such that $x^m=f$. Then $e=e^m=(x^{n})^m=x^{nm}=(x^m)^n=f^n=f$ which implies that the monogenic subsemigroup $x^\IN=\{x^n:n\in\IN\}$ of $S$ contains a unique idempotent. This allows us to define a function $\pi:S\to E(S)$ assigning to each $x\in S$ the unique idempotent of the semigroup $x^\IN$.

If some element $x\in S$ belongs to a maximal subgroup $H_e\subseteq S$, then $\pi(x)\in x^\IN\subseteq H_e$ coincides with the unique idempotent $e$ of the group $H_e$ and hence $\pi[H_e]=\{e\}$. Therefore, $\pi[H(S)]=E(S)=\pi[S]$.

For a semigroup $S$ let
$$Z(S)=\{z\in S:\forall x\in S\;\;(xz=zx)\}$$be the {\em center} of $S$.

\begin{proposition}\label{p:pi-homo} If $S$ is a periodic semigroup with $E(S)\subseteq Z(S)$, then $\pi:S\to E(S)$ is a homomorphism and $H(S)$ is a subsemigroup of $S$.
\end{proposition}

\begin{proof} Since $E(S)\subseteq Z(S)$, the set $E(S)$ is a semilattice and hence $E(S)$ carries the partial order $\le$ defined by $x\le y$ if and only if $xy=x$.

\begin{claim}\label{cl:pi-Z} For any $x\in S$ and $y\in Z(S)$ we have $\pi(xy)=\pi(x)\pi(y)$.
\end{claim}

\begin{proof} Since $S$ is periodic, there exist numbers $n,m\in\IN$ such that $\pi(x)=x^n$ and $\pi(y)=y^m$. Since $xy=yx$, $(xy)^{nm}=x^{nm}y^{mn}=\pi(x)^m\pi(y)^n=\pi(x)\pi(y)\in E(S)$ and hence $\pi(xy)=\pi(x)\pi(y)$.
\end{proof}

The following claim implies that $H(S)$ is a subsemigroup of $S$.

\begin{claim}\label{cl:Hprod} For any $x,y\in E(S)$ we have $H_xH_y\subseteq H_{xy}$.
\end{claim}

\begin{proof} Take any elements $a\in H_x$ and $b\in H_y$ and observe that since $x,y\in E(S)\subseteq Z(S)$ we have  $abS^1=ayS^1=yaS^1=yxS^1=xyS^1$ and $S^1ab=S^1xb=S^1bx=S^1yx=S^1xy$ and hence $ab\in H_{xy}$.
\end{proof}

\begin{claim}\label{cl:ineq} For any $x,y\in S$ we have $\pi(x)\pi(y)\le\pi(xy)$.
\end{claim}

\begin{proof} By Lemma~\ref{l:C-ideal}, $x\pi(x)\in \korin{\infty}{\pi(x)}\pi(x)\in H_{\pi(x)}$ and $y\pi(y)\in H_{\pi(y)}$. Then $xy\pi(x)\pi(y)=x\pi(x)y\pi(y)\in H_{\pi(x)}H_{\pi(y)}\subseteq H_{\pi(x)\pi(y)}$ according to Claim~\ref{cl:Hprod}. Hence $\pi(x\pi(x)y\pi(y))=\pi(x)\pi(y)$. By Claim~\ref{cl:pi-Z},
$$\pi(x)\pi(y)=\pi(x\pi(x)y\pi(y))=\pi(xy\pi(x)\pi(y))=\pi(xy)\pi(x)\pi(y),$$
which means that $\pi(x)\pi(y)\le\pi(xy)$.
\end{proof}

\begin{claim}\label{cl:SxH} For any $x\in S$ and $y\in H(S)$ we have $\pi(xy)=\pi(x)\pi(y)$.
\end{claim}

\begin{proof} It follows from $y\in H(S)$ that $y\in H_{\pi(y)}$ and hence $y=y\pi(y)$. Let $y^{-1}$ be the inverse element of $y$ in the group $H_{\pi(y)}$.
By Claims~\ref{cl:ineq} and \ref{cl:pi-Z},
$$\pi(x)\pi(y)\le \pi(xy)=\pi(xy\pi(y))=\pi(xy)\pi(y)=\pi(xy)\pi(y^{-1})\le \pi(xyy^{-1})=\pi(x\pi(y))=\pi(x)\pi(y)$$ and hence $\pi(xy)=\pi(x)\pi(y)$.
\end{proof}

\begin{claim}\label{cl:xy=yx} For every $x,y\in S$ we have $\pi(xy)=\pi(yx)$.
\end{claim}

\begin{proof} Since $S$ is periodic, there exists $n\in\IN$ such that $(xy)^n$ and $(yx)^n$ are idempotents. Taking into account that $E(S)\subseteq Z(S)$, we conclude that
\begin{multline*}
(xy)^n=((xy)^n)^{n+1}=((xy)^{n+1})^n=(x(yx)^ny)^n=((yx)^n)^n(xy)^n=(yx)^n(xy)^n=\\
(yx)^n((xy)^n)^n=(y(xy)^nx)^n=((yx)^{n+1})^n=((yx)^n)^{n+1}=(yx)^n
\end{multline*}
and hence $\pi(xy)=(xy)^n=(yx)^n=\pi(yx)$.
\end{proof}

\begin{claim} For every $x,y\in S$ we have $\pi(xy)=\pi(x)\pi(y)$.
\end{claim}

\begin{proof} By Claim~\ref{cl:xy=yx}, $\pi(xy)=\pi(yx)$. Let $e=\pi(xy)=\pi(yx)$.
By Claim~\ref{cl:ineq}, $\pi(x)\pi(y)\le\pi(xy)=e$. Since $S$ is periodic, there exists $n\in\IN$ such that $(xy)^n=e=(yx)^n$.
Observe that $xeS^1=exS^1\subseteq eS^1$ and $eS^1=eeS^1=(xy)^neS^1\subseteq xeS^1$ and hence $xeS^1=eS^1$. Similarly, $S^1xe\subseteq S^1e$ and $S^1e=S^1ee=S^1e(yx)^n\subseteq S^1ex=S^1xe$ and hence $xe\in H_e$. By analogy we can prove that $ye\in H_e$. By Claim~\ref{cl:Hprod} and the inequality $\pi(x)\pi(y)\le e$, we finally have
$$\pi(x)\pi(y)=\pi(x)\pi(y)e=\pi(xe)\pi(ye)=\pi(xeye)=e=\pi(xy).$$
\end{proof}
\end{proof}

\section{Sufficient conditions of $\C$-closedness}\label{s:cool}

In this section we prove some sufficient conditions of the $\C$-closedness of a semigroup. We start with the following lemma that implies Theorem~\ref{cool}, announced in the introduction.

\begin{lemma}\label{l:cool} A subsemigroup $X$ of a topological semigroup $Y$ is closed in $Y$ if $X$ admits a continuous homomorphism $h:X\to E$ to a chain-finite discrete topological semilattice $E$ such that for every $e\in E$ the set $h^{-1}(e)$ is closed in $Y$.
\end{lemma}

\begin{proof}  To derive a contradiction, assume that $X$ is not closed in $Y$. So, we can fix an element $y\in \overline X\setminus X\subseteq Y$. Replacing $Y$ by $\overline X$, we can assume that $X$ is dense in $Y$.

\begin{claim}\label{cl} If for some $a\in X^1$, $e\in E$, and
$n\in\IN$ we have $ay^n\in h^{-1}(e)$, then $ay\in
h^{-1}(e)$.
\end{claim}

\begin{proof} Since $h^{-1}(e)$ is an open subspace of $X$, there exists an open subset $W\subseteq Y$ such that $W\cap X=h^{-1}(e)$. Assuming that $ay^n\in h^{-1}(e)\subseteq W$, we can find a
neighborhood $V\subseteq Y$ of $y$ such that $
aV^n\subseteq W$. Then for every $v\in X\cap V$ we have
$h(av)=h(a)h(v)=h(a)h(v)^n=h(av^n)=h(ay^n)=e$ and
hence $ay\in a(\overline{X\cap V})\subseteq \overline{a(X\cap
V)}\subseteq\overline{h^{-1}(e)}=h^{-1}(e)$.
\end{proof}

\begin{claim}\label{cl2} If $ay\in h^{-1}(e)$ for some $a\in X^1$ and $e\in E$, then the point $y$ has a neighborhood $U\subseteq Y$ such that $aU\subseteq h^{-1}(e)$.
\end{claim}

\begin{proof} Since the set $h^{-1}(e)$ is open in $X$, there exists an open set $W$ in $Y$ such that $h^{-1}(e)=X\cap W$. Since $ay\in h^{-1}(e)\subseteq W$, there exists a neighborhood $U\subseteq Y$ of $y$ such that $aU\subseteq W$. Then $a(U\cap X)\subseteq W\cap X=h^{-1}(e)$ and $aU\subseteq a(\overline{U\cap X})\subseteq\overline{a(U\cap X)}\subseteq\overline{h^{-1}(e)}=h^{-1}(e)$.
\end{proof}

Let $\Tau_y$ be the family of all neighborhoods of $y$ in $Y$.
In the semilattice $E$ consider the subset
 $$E_y=\{e\in E:\exists
U\in\Tau_y\;\;h[X\cap U]\subseteq{\uparrow}e\}, \hbox{ where } {\uparrow}e=\{f\in E: e\le f\}.$$

 The set $E_y$
contains the smallest element of the chain-finite semilattice $E$ and hence $E_y$ is
not empty. Let $e$ be a maximal element of $E_y$ (which exists, because $E$ is chain-finite) and $W\in \Tau_y$
be a neighborhood of $y$ such that $h[X\cap
W]\subseteq{\uparrow}e$.

 By induction we shall construct a sequence $(v_n)_{n\in\w}$ of
points of $W\cap X$ such that for every $n\in\w$ the following
conditions are satisfied:
 \begin{enumerate}
 \item $v_0\cdots v_ny\notin h^{-1}(e)\cup
 h^{-1}(h(v_0\cdots v_n))$;
 \item $e<h(v_0\cdots v_{n+1})<h(v_0\cdots v_{n})$.
 \end{enumerate}

To start the inductive construction, observe that Claim~\ref{cl}
implies $y^2\notin h^{-1}(e)$, so we can find a neighborhood
$V\subseteq W$ of $y$ such that $VV\cap h^{-1}(e)=\emptyset$.
Choose any element
  $v_0\in V\cap X$ and observe that $v_0y\in VV\subseteq Y\setminus
h^{-1}(e)$. Taking into account that
$h(v_0)=h(v_0)h(v_0)=h(v_0v_0)\in h[X\cap VV]\subseteq
E\setminus \{e\}$ and $v_0\in V\cap X\subseteq
W\cap X\subseteq h^{-1}[{\uparrow}e]$, we conclude that $h(v_0)>e$.

We claim that $v_0y\notin h^{-1}(h(v_0))$. Assuming that
$v_0y\in h^{-1}(h(v_0))\subseteq X$, we can apply Claim~\ref{cl2} and find a neighborhood
$U\subseteq V$ of $y$ such that $v_0U\subseteq h^{-1}(h(v_0))$. By the
maximality of $e$, the set $h[X\cap U]$ is not contained in
$h^{-1}[{\uparrow}h(v_0)]$, so we can find an element $u\in
X\cap U$ such that $h(u)\notin{\uparrow}h(v_0)$. Then
$h(v_0)\ne h(v_0)\cdot h(u)=h(v_0u)=h(v_0)$,
which is a desired contradiction.

 Now assume that for some $n\in\IN$ the points $v_0,\dots,
v_{n-1}\in W\cap X$ with $v_0\dots v_{n-1}y\notin
 h^{-1}(e)\cup h^{-1}(h(v_0\cdots v_{n-1}))$ and $h(v_0\cdots
v_{n-1})>e$ have been constructed. Claim~\ref{cl} implies
$v_0\dots v_{n-1}y^2\notin h^{-1}(e)\cup h^{-1}(h(v_0\cdots
v_{n-1}))$.

 Since the semigroups $h^{-1}(e)$ and $h^{-1}(h(v_0\cdots
v_{n-1}))$ are closed in $Y$, we can find a neighborhood $V\subseteq
W$ of $y$ such that the set $v_0\dots v_{n-1}VV$ is disjoint with
the closed set $h^{-1}(e)\cup h^{-1}(h(v_0\cdots v_{n-1}))$.
Choose any element $v_n\in V\cap X$. The choice of $W$ guarantees that $v_n\in V\cap X\subseteq W\cap X$ and hence $h(v_n)\in h[W\cap X]\subseteq{\uparrow}e$ and $h(v_0\cdots v_n)=h(v_0\cdots v_{n-1})\cdot h(v_n)\in{\uparrow}e\cdot{\uparrow} e={\uparrow} e$. On the other hand, $v_0\dots
v_{n-1}v_ny\in v_0\dots v_{n-1}VV$ and hence $v_0\cdots v_ny$ does
not belong to the set $h^{-1}(e)\cup h^{-1}(h(v_0\cdots
v_{n-1}))$. Also the idempotent
$$h(v_0\cdots
v_n)=h(v_0\cdots v_{n-1}v_n^2)\in h[X\cap v_0\cdots
v_{n-1}VV]$$does not belong to the set $\{e\}\cup\{h(v_0\cdots
v_{n-1})\}$, which implies that $e<h(v_0\cdots v_n)<h(v_0\cdots
v_{n-1})$.

 Finally, we show that $v_0\cdots v_ny\notin h^{-1}(h(v_0\cdots
v_n))$. Assuming the opposite and using Claim~\ref{cl2},  we can find a neighborhood $U\subseteq W$ of $y$ such
that $v_0\cdots v_nU\subseteq h^{-1}(h(v_0\cdots v_n))$. Since
$h(v_0\cdots v_n)>e$, the maximality of the element $e$
guarantees that $h[X\cap U]\not\subseteq {\uparrow}h(v_0,\dots
v_n)$, so we can choose an element $u\in X\cap U\setminus
h^{-1}[{\uparrow}h(v_0\cdots v_n)]$ and conclude that
$h(v_0\cdots v_nu)=h(v_0\cdots v_n)h(u)<h(v_0\cdots v_n)$,
which contradicts $v_0\cdots v_nu\in v_0\cdots
v_nU\subseteq h^{-1}(h(v_0\dots v_n))$.
 This contradiction completes the inductive step.

 After completing the inductive construction, we obtain a strictly
decreasing sequence $\big(h(v_0\cdots v_n)\big)_{n\in\w}$ of
idempotents in $E$, which is not possible as $E$ is chain-finite.
\end{proof}

Next we prove a sufficient condition of the $\C$-closedness of a bounded semigroup.

\begin{lemma}\label{l:bounded} A bounded semigroup $X$ is $\C$-closed if  $E(X)$ is a $\C$-closed semigroup and for every infinite set $A\subseteq X$ the set $AA$ is not a singleton.
\end{lemma}

\begin{proof} Assuming that the semigroup $X$ is not $\C$-closed, we can find an isomorphic topological embedding $h:X\to Y$ of $X$ endowed with the discrete topology to a topological semigroup $(Y,\tau)\in\C\subseteq \mathsf{T_{\!1}S}$. By our assumption, the set $h[E(X)]$ is closed in $Y$, being a $\C$-closed semigroup. Being discrete, the subspace $h[X]$ is open in its closure $\overline{h[X]}$. Identifying $X$ with its image $h[X]$ and replacing $Y$ by $\overline{h[X]}$, we conclude that $X$ is a dense open discrete subsemigroup of a topological semigroup $Y\in\mathsf{T_{\!1}S}$ such that $E(X)$ is closed in $Y$.

Since $X$ is bounded, there exists $n\in\IN$ such that for every
$x\in X$ the power $x^n$ is an idempotent of $X$.
Pick any point
$a\in Y\setminus X$. Note that
$a^n\in\overline{\{x^n:x\in X\}}\subseteq \overline{E(X)}=E(X)$. Let $e=a^n\in E(X)$.  By the continuity of the semigroup operation,
the point $a$ has a neighborhood $O_a\subseteq Y$ such that
$O_a^n=\{e\}$. Let $H_e$ be the maximal subgroup of $Y$,
containing
$e$. Then $X\cap H_e$ is the maximal subgroup of $X$ containing $e$. For
every $x\in O_a\cap X$ we get $x^n=e$ and hence $x^m\in X\cap H_e$ for
all $m\ge n$ (see Lemma~\ref{l:C-ideal}). We claim that for any $m\ge n$ the element $a^m$
belongs to the semigroup $X$. Taking into account that
$(a^m)^n=(a^n)^m=e^m=e$, we can find a neighborhood $U\subseteq Y$
of
$a^m$ such that $U^n=\{e\}$. Next, find a neighborhood $V\subseteq
O_a$ of $a$ such that $V^m\subseteq U$. It follows that $a^m$ is
contained in the closure of the set $W:=\{v^m:v\in V\cap
X\}\subseteq X\cap H_e$. Assuming that $a^m\notin X$, we conclude that the set
$W\subseteq U\cap X\cap H_e$ is infinite. Since $W$ is a subset of the
group $X\cap H_e$, the product $W^n\subseteq U^n$ is infinite and cannot
be
equal to the singleton $U^n=\{e\}$. This contradiction shows that
$a^m\in X$ for all $m\ge n$. Then there exists a number $k\in\w$
such that $a^{2^k}\notin X$ but $a^{2^{k+1}}\in X$. By the
continuity of the semigroup operation, the point $b=a^{2^k}$ has a
neighborhood $O_b\subseteq Y$ such that $O_b^2=\{b^2\}\subseteq X$.
Since $b\in\overline X\setminus X$, the set $A=O_b\cap X$ is infinite
and $AA\subseteq O_b^2=\{b\}$ is a singleton.
\end{proof}



Finally we establish a sufficient condition of the $\C$-closedness of a periodic commutative semigroup with a unique idempotent.

\begin{lemma}\label{l:single} A periodic commutative semigroup $X$ with a unique idempotent $e$ is $\mathsf{T_{\!1}S}$-closed if the maximal subgroup $H_e$ of $X$ is bounded and for every infinite set $A\subseteq X$ the set $AA$ is not a singleton.
\end{lemma}

\begin{proof} Assume that  the maximal subgroup $H_e$ of $X$ is bounded and for every infinite set $A\subseteq X$ the set $AA$ is not a singleton. To derive a contradiction, assume that $X$ is not $\mathsf{T_{\!1}S}$-closed and hence $X$ is a non-closed discrete subsemigroup of some topological semigroup $(Y,\tau)\in\mathsf{T_{\!1}S}$. Replacing $Y$ by the closure of $X$, we can assume that $X$ is dense and hence open in $Y$.

\begin{claim}\label{cl:Pe-ideal} The semigroup $X$ is an ideal in $Y$.
\end{claim}

\begin{proof} Given any elements $x\in X$ and $y\in Y$, we should prove that $xy\in X$. Since $X$ is periodic, there exists $n\in\IN$ such that $x^n=e$. Consider the set $\korin{n}{H_e}=\{b\in X:b^n\in H_e\}$. We claim that $\korin{n}{H_e}$ is an ideal in $X$. Indeed, for any $b\in \korin{n}{H_e}$ and $z\in X$ we have $(bz)^n=b^nz^n\in H_ez^n\subseteq H_e$ as $H_e$ is an ideal in $X$ (see Lemma~\ref{l:C-ideal}). Since the group $H_e$ is bounded, the semigroup $\korin{n}{H_e}$ is bounded, too. By Lemma~\ref{l:bounded}, the bounded semigroup $\korin{n}{H_e}$ is $\mathsf{T_{\!1}S}$-closed and hence closed in $Y$.

Taking into account that $\korin{n}{H_e}$ is an ideal in $X$ and $x\in \korin{n}{H_e}$, we conclude that $$xY=x\overline{X}\subseteq \overline{xX}\subseteq \overline{\korin{n}{H_e}\cdot X}\subseteq\overline{\korin{n}{H_e}}=\korin{n}{H_e}\subseteq X.$$
\end{proof}

Take any point $y\in Y\setminus X$ and consider its orbit $y^\IN=\{y^n:n\in\IN\}$.

\begin{claim}\label{cl:yN} $y^\IN\cap X=\emptyset$.
\end{claim}

\begin{proof} To derive a contradiction, assume that $y^n\in X$ for some $n\in\IN$. We can assume that $n$ is the smallest number with this property. It follows that $n\ge 2$ and hence $2n-2\ge n+(n-2)\ge n$. Then $y^{n-1}\notin X$ and $y^{2n-2}=y^ny^{n-2}\subseteq XY^1\subseteq X$, because $X$ is an ideal in $Y$. Since $X$ is an open discrete subspace of the topological semigroup $(Y,\tau)$, there exists a neighborhood $U\in\tau$ of $y$ such that $U^{2n-2}=\{y^{2n-2}\}$. Consider the set $A=(U\cap X)^{n-1}$ and observe that $AA\subseteq U^{2n-2}=\{y^{2n-2}\}$ is a singleton. On the other hand, $y^{n-1}\in \overline A\setminus X$ which implies that the set $A$ is infinite. But the existence of such set $A$ contradicts our assumptions.
\end{proof}

By our assumption, the maximal subgroup $H_e$ of $X$ is bounded  and hence there exists $p\in\IN$ such that $x^p=e$ for every $x\in H_e$.
Consider the subsemigroup $P=\{x^p:x\in X\}$ in the semigroup $X$. It follows from $y\in \overline{X}$ that the element $y^p$ belongs to the closure of the set $P$ in $Y$.

\begin{claim}\label{cl:Pe=e} $Pe=\{e\}$.
\end{claim}

\begin{proof} Given any element $x\in P$, find an element $z\in X$ such that $x=z^p$. By Lemma~\ref{l:C-ideal}, the subgroup $H_e$ is an ideal in $X$, which implies $ze\in H_e$. The choice of $p$ ensures that $xe=z^pe^p=(ze)^p=e$.
\end{proof}

\begin{claim}\label{cl:nmy} For any $x\in P$ there exists $n\in \IN$ such that $x(y^p)^m=e$ for all $m\ge n$.
\end{claim}

\begin{proof} Since $X$ is an ideal in $Y$ we have $xy^p\in X$. Since $X$ is an open discrete subspace of the topological semigroup $(Y,\tau)$, there exists a neighborhood $U\subseteq Y$ of $y$ such that $xU^p=\{xy^p\}$. Choose any element $u\in U\cap X$ and observe that $xU^p=\{xu^p\}=\{xy^p\}$. Let $V=\{v^p:v\in U\cap X\}\subseteq P$ and observe that $xV=\{xu^p\}$.  Then $xVV=xu^pV=u^pxV=u^pxu^p=xu^{2p}$. Proceeding by induction, we can show that
$xV^n=xu^{np}$ for every $n\in\IN$. Since the semigroup $X$ is periodic, there exists $n\in\IN$ such that $u^{pn}=e$. Then for every $m\ge n$, we obtain
$$xV^m=xu^{mp}=xu^{np}u^{p(m-n)}\in xeP=\{e\}$$by Claim~\ref{cl:Pe=e}.
Then
$$x(y^p)^m\in x\overline V^m\subseteq\overline{xV^m}=\overline{\{e\}}=\{e\}.$$
\end{proof}

Now we are able to finish the proof of Lemma~\ref{l:single}. Inductively we shall construct sequences of points $(x_k)_{k\in\IN}$ in $X$,  positive integer numbers $(n_k)_{k\in\IN}$, $(m_k)_{k\in\IN}$ and open neighborhoods $(U_k)_{k\in\w}$ of $y$ in $Y$ such that for every $k\in\IN$ the following conditions are satisfied:
\begin{enumerate}
\item[(i)] $x_k\in U_{k-1}$;
\item[(ii)]  $x_k^{pm_k}\notin\{e\}\cup\{x_i^{pm_i}:i<k\}$, $x_k^{2pm_k}=e$, and $m_k>n_{k-1}$;
\item[(iii)] $x_k^p(y^p)^{n_k}=e$, $x_k^pU_k^{pn_k}=\{e\}$, and $n_k>n_{k-1}$;
\item[(iv)] $y\in U_k\subseteq U_{k-1}$.
\end{enumerate}

To start the inductive construction, choose any neighborhood $U_0\subseteq Y$ of $y$ such that $e\notin U_0^p$. Such neighborhood exists since $e\ne y^p$ by Claim~\ref{cl:yN}. Also put $n_0=2$. Now assume that for some $k\in\IN$ and all $i<k$ we have constructed a point $x_i$, a neighborhood $U_i$ of $y$ and two numbers $n_i,m_i$ satisfying the inductive conditions. Since $y^\IN\cap X=\emptyset$, there exists a neighborhood $W\subseteq Y$ of $y$ such that $e\notin W^l$ for any natural number $l\le p(2+2k+n_{k-1})$.
Choose any point $x_k\in U_{k-1}\cap W\cap X$. Then $e\ne x_k^l$ for any  natural number $l\le p(2+2k+n_{k-1})$.
Since the semigroup $X$ is periodic and has a unique idempotent $e$, there exists a number $l_k$ such that $(x_k^p)^{l_k+k+1}=e$. We can assume that $l_k$ is the smallest number with this property. Then $(x_k^p)^{l_k+k}\ne e$. The choice of the neighborhood $W\ni x_k$ ensures that $l_k> n_{k-1}+k+1$.

\begin{claim}\label{cl:distinct} The set $\{(x_k^p)^{l_k+i}:0\le i\le k\}$ has cardinality $k+1$.
\end{claim}

\begin{proof} Assuming that this set has cardinality smaller than $k+1$, we can find two numbers $i,j$ such that $l_k\le i<j\le l_k+k$ and $x_k^{pi}=x_k^{pj}$. The equality $x_k^{pi}=x_k^{pj}=x_k^{pi}x_k^{p(j-i)}$ implies $x_k^{pi}=x_k^{pi}x_k^{np(j-i)}$ for all $n\in\IN$.  Find a (unique) number $n\in\IN$ such that $i\le n(j-i)<j$. Then
$$x_k^{pn(j-i)}x_k^{pn(j-i)}=x_k^{pn(j-i)-pi}x_k^{pi}x_k^{pn(j-i)}=x_k^{pn(j-i)-pi}x_k^{pi}=x_k^{pn(j-i)}$$and hence $x_k^{pn(j-i)}$ is an idempotent. Since the semigroup $X$ contains a unique idempotent, $x_k^{pn(j-i)}=e$ and hence $j> n(j-i)\ge l_k+k+1$, which contradicts the choice of $j$.
\end{proof}

By Claim~\ref{cl:distinct}, there exists a number $j$ such that $0\le j\le k$ and $(x_k^p)^{l_k+j}\notin\{e\}\cup\{x_i^{pm_i}:1\le i<k\}$. Put $m_k=l_k+j$ and observe that
$$
(x_k^p)^{2m_k}=(x_k^p)^{l_k+k+1}(x_k^p)^{l_k+2j-k-1}\in eP=\{e\}$$by Claim~\ref{cl:Pe=e}.
By Claim~\ref{cl:nmy}, there exists a number $n_k>n_{k-1}$ such that $x_k^p(y^p)^{n_k}=e$. Since $X$ is an open discrete subspace of the topological semigroup $Y$, there exists a neighborhood $U_k\subseteq U_{k-1}$ of $y$ such that $x_k^p(U_k)^{pn_k}=\{e\}$. This completes the inductive construction.

Now consider the set $A=\{x_k^{pm_k}:k\in\IN\}$ of $P$.  The inductive condition (ii) guarantees that $A$ is infinite and $a^2=e$ for every $a\in A$. Also for any $i<j$ we have
$$x_i^{pm_i}x_j^{pm_j}= x_i^px_j^{pn_i}x_i^{p(m_i-1)}x_j^{p(m_j-n_i)}\in x_i^p(U_i)^{pn_i}P=eP=\{e\}.$$Therefore, $AA=\{e\}$ is a singleton. But the existence of such set $A$ is forbidden by our assumption.
\end{proof}

\section{Some properties of  $\Zero$-closed semigroups}\label{s:Zero}



\begin{lemma}\label{l:chain} For every $\Zero$-closed semigroup $X$, its center $Z(X)$ is chain-finite.
\end{lemma}

\begin{proof} To derive a contradiction, assume that the
semigroup $Z(X)$ contains an infinite chain $C$. Take
any free ultrafilter $\U\in\beta(X)\subseteq\Fil(X)$ containing the
set $C$. Since $C$ is a chain, for every set $U\subseteq C$ we have $UU=U$, which implies that $\U\U=\U$. Let $Y$
be
the smallest subsemigroup of the semigroup $\Fil(X)$, containing
the set $X\cup\{\U\}$. Since the set $C$ is contained in the
center
of the semigroup $X$ and $\U\U=\U$, the semigroup $Y$ is
equal to the set $X\cup\{x\U:x\in X^1\}\subseteq\beta(X)$.
Then $X$ is not $\Zero$-closed, being a proper dense subsemigroup
of the Hausdorff zero-dimensional topological semigroup $Y$.
\end{proof}

\begin{corollary}\label{BBT1}
For a semilattice $X$ the following conditions are equivalent:
\begin{enumerate}
\item $X$ is projectively $\C$-closed;
\item $X$ is $\C$-closed;
\item $X$ is chain-finite.
\end{enumerate}
\end{corollary}

\begin{proof}
The implication $(1)\Rightarrow (2)$ is trivial, the implication $(2)\Rightarrow (3)$ follows from Lemma~\ref{l:chain}. To prove that $(3)\Ra(1)$, assume that the semilattice $X$ is chain-finite. First we show that a homomorphic image of  $X$ is chain-finite. Assuming the contrary, pick a homomorphism $h: X\rightarrow Y$ such that the semilattice $h[X]$ contains an infinite chain $L$. Let $\{y_n:n\in\omega\}$ be an infinite subset of $L$. For each $i\in\omega$, $h^{-1}(y_i)$ is a subsemilattice of $X$. Since $X$ is chain-finite, so is the semilattice $h^{-1}(y_i)$, $i\in\omega$. Since each chain-finite semilattice has the smallest element, for each $i\in\omega$ we can consider the element $x_i=\inf h^{-1}(y_i)\in h^{-1}(y_i)$. We claim that the set $K=\{x_i:i\in\omega\}$ is an infinite chain in $X$. Clearly, $K$ is infinite. Fix any $i,j\in\omega$. With no loss of generality we can assume that $y_iy_j=y_j$. Then
 $$x_ix_j\in h^{-1}(h(x_i x_j))=h^{-1}(h(x_i)h(x_j))=h^{-1}(y_i y_j)=h^{-1}(y_j).$$
 Since $x_ix_j\leq x_j$ and $x_j=\inf h^{-1}(y_j)$ the formula above implies that $x_ix_j=x_j$, witnessing that $K$ is an infinite chain which contradicts the chain-finiteness of $X$.
 At this point the implication $(3)\Ra(1)$ follows from Lemma~\ref{l:cool}.
\end{proof}

\begin{lemma}\label{l:AA} If a semigroup $X$ is $\Zero$-closed, then for any infinite subset $A\subseteq Z(X)$ the set $AA$ is not a singleton.
\end{lemma}

\begin{proof} Assume that for some infinite set $A\subseteq Z(X)$ the product $AA$ is a singleton. Choose any free ultrafilter $\U\in\beta(X)$ containing
the
set $A$ and observe that $\U\U$ is a principal ultrafilter
(containing the singleton $AA$). Then the subsemigroup $Y
\subseteq\Fil(X)$ generated by the set $X\cup\{\U\}$ is equal to
$X\cup\{x\U :x\in X^1\}$ and hence is contained in $\beta(X)$.
Consequently, $X$ is a non-closed subsemigroup of Hausdorff zero-dimensional topological semigroup
$Y$, which means that $X$ is not $\Zero$-closed.
\end{proof}



\begin{lemma}\label{l:period} The center $Z(X)$ of any
$\Zero$-closed semigroup $X$ is periodic.
\end{lemma}

\begin{proof} Assuming that $Z(X)$ is not periodic, find $x\in
Z(X)$ such that the powers $x^n$,
$n\in\IN$, are pairwise distinct. On the set $X$ consider the free
filter $\F$ generated by the base consisting of the sets
$x^{n!\IN}=\{x^{n!k}:k\in\IN\}$, $n\in\IN$.

Taking into account that $(n+1)!\IN\subseteq n!\IN+n!\IN\subseteq
n!\IN$ for all $n\in\IN$, we conclude that    $\F\F=\F$, so $\F$
is
an idempotent of the semigroup $\Fil(X)$. Let $Y
=X\cup\{a\F:a\in X^1\}$ be the smallest subsemigroup of $\Fil(X)$
containing the set $X\cup\{\F\}$.  We endow $Y$ with the
subspace topology inherited from $\Fil(X)$. Then $Y$ is a
topological semigroup, containing $X$ as a proper dense discrete
subsemigroup. Since the space $\Fil(X)$ is $T_0$ it is sufficient to show that the space $Y$ is zero-dimensional, because zero-dimensional $T_0$ spaces are Hausdorff.

By $I$ denote the set of all elements $a\in X^1$ such that the
function $\IN\to X$, $n\mapsto ax^n$, is injective. It is clear
that for every $a\in I$ the filter $a\F$ is free and hence does
not
belong to the set $X\subset Y$ of principal ultrafilters.

\begin{claim}\label{cl:I} For any $a\in X\setminus I$ the filter
$a\F$ is principal.
\end{claim}

\begin{proof} By the definition of the set $I$, there are two
numbers $n,k\in\IN$ such that $ax^n=ax^{n+k}=ax^nx^k$ and hence
$ax^n=ax^nx^{ki}$ for all $i\in\IN$.
Find a number $j\in\IN$ such that $0\le kj-n<k$ and observe that
for every integer number $i>j$ we get
$ax^{ki}=ax^{kj-n}x^{n}x^{k(i-j)}=x^{kj-n}ax^{n}x^{k(i-j)}=x^{kj-n}ax^n=ax^{kj}$.
Consequently, for the set $F=\{x^{ki}:i>j\}\in\F$ the set
$aF=\{ax^{kj}\}$ is a singleton, which implies that
the filter $a\F$ is principal.
\end{proof}

\begin{claim} The topological semigroup $Y$ is
zero-dimensional.
\end{claim}

\begin{proof} We need to show that for any point $y\in Y$,
any neighborhood $O_y\subseteq Y$ of $y$ contains a clopen
neighborhood of $y$.
If $y\in X$, then $y$ is an isolated point of the space
$Y$ and $\{y\}$ is an open neighborhood of $y$, contained
in $O_y$. To show that the set $\{y\}$ is closed in $Y$ fix any $z\in Y\setminus \{y\}$.
Claim~\ref{cl:I} implies that any element of $Y$ is either a principal ultrafilter or a free filter on $X$. Anyway, there exists $T\in z$ such that $y\notin T$. It is easy to see that $\langle T\rangle$ is an open neighborhood of $z$ disjoint with $\{y\}$.


 Next, assume that $y\notin X$ and hence $y=a\F$ for some
$a\in X^1$. By Claim~\ref{cl:I}, $a\in I$. Find a set
$F=x^{k!\IN}\in\F$ such that $\langle aF\rangle\subseteq O_y$. We claim
that the basic open set $\langle aF\rangle$ is closed in $Y$. Given any point $t\in Y\setminus\langle aF\rangle$, we
should find a neighborhood $O_t\subseteq Y$, which is disjoint
with $\langle aF\rangle$. If $t\in X$, then the neighborhood
$O_t=\{t\}$ of $t$ is disjoint with $\langle aF\rangle$ and we are
done. So, we assume that $t\notin X$. In this case $t=b\F$ for
some $b\in I$, according to Claim~\ref{cl:I}.

We claim that $aF\cap bF=\emptyset$. To derive a contradiction,
assume that $aF\cap bF$ contains some common element
$ax^{k!n}=bx^{k!m}$ where $n,m\in\IN$. Then
$ax^{k!(n+i)}=bx^{k!(m+i)}$ for all $i\in\IN$ and hence the symmetric difference $aF\triangle bF=(aF\setminus bF)\cup (bF\setminus aF)\subseteq\{ax^{k!i}\}_{i\le n}\cup\{bx^{k!j}\}_{j\le m}$ is finite. Since the filter $t$ is free and $bF\in t$ we obtain that  $aF\in t$, which contradicts
the choice of $t$.

This contradiction shows that $aF\cap bF=\emptyset$ and hence
$\langle bF\rangle$ is a neighborhood of the filter $t$, disjoint
with the set $\langle aF\rangle$, which implies that $\langle
aF\rangle$ is clopen and the space $Y$ is
zero-dimensional.
\end{proof}

Therefore, $X$ is not $\Zero$-closed.
\end{proof}


\begin{lemma}\label{l:subgrp2} Assume that $X$ is a periodic $\Zero$-closed semigroup with $H(X)\subseteq Z(X)$. If $X$ contains an unbounded subgroup, then for some $e\in E(X)$ and $x\in X$ there exists an infinite set $A\subseteq x\cdot H_e$ such that $AA$ is a singleton.
\end{lemma}

\begin{proof} To derive a contradiction, assume that $X$ contains an unbounded subgroup but for any $e\in E(X)$, $x\in X$ and an infinite set $A\subseteq x\cdot H_e$ the set $AA$ is not a singleton.

Since $E(X)\subseteq H(X)\subseteq Z(X)$, the set of idempotents $E(X)$ is a semilattice. Let $\pi:X\to E(X)$ be the map assigning to each $x\in X$ the unique idempotent of the monogenic semigroup $x^\IN$. By Proposition~\ref{p:pi-homo}, $\pi$ is a homomorphism.

Since $X$ contains an unbounded subgroup, for some idempotent $e\in E(X)$ the maximal subgroup $H_e$ containing $e$ is unbounded.
By Lemma~\ref{l:chain}, the semilattice $E(X)$ is chain-finite.
Consequently, we can find an
idempotent $e$ whose maximal group $H_e$ is unbounded but for
every idempotent $f<e$ the group $H_{f}$ is bounded.


In the semigroup $X$, consider the set $$T=\textstyle\bigcup\big\{\!\korin{\infty}{H_f}:f\in E(X),\;fe<e\big\}.$$ 

\begin{claim}\label{cl:quot} For every $a\in T$, the set $G_a=\{x\in H_e:ax=ae\}$ is a subgroup of $H_e$ such that the quotient group $H_e/G_a$ is bounded.
\end{claim}

\begin{proof} Observe that for any $x,y\in G_a$ we have $axy=aey=ay=ae$, which means that $G_a$ is a subsemigroup of the group $H_e$. Since the group $H_e$ is periodic, the subsemigroup $G_a$ is a subgroup of $H_e$. It remains to prove that the quotient group $H_e/G_a$ is bounded. To derive a contradiction, assume that $H_e/G_a$ is unbounded.

Let $f=\pi(a)$. It follows from $a\in T$ that $fe<e$. Now the minimality of $e$ ensures that the group $H_{fe}$ is bounded. Then there exists $p\in\IN$ such that $x^p=fe$ for any $x\in H_{fe}$.

\begin{claim}\label{cl:xhp} For every $x\in H_{fe}$ and $h\in H_e$ we have $xh^p=x$.
\end{claim}

\begin{proof} By Proposition~\ref{p:pi-homo}, $\pi(feh)=fe\pi(h)=fee=fe$ and by Lemma~\ref{l:C-ideal},
$$feh=feh\cdot fe=feh\cdot\pi(feh)\in H_{\pi(feh)}=H_{fe}.$$ Then $(feh)^p=fe$ and
$xh^p=(xfe)h^p=x(feh)^p=xfe=x.$
\end{proof}

In the group $H_e$ consider the subgroup $G=\{h^p:h\in H_e\}$. By Proposition~\ref{p:pi-homo}, $\pi(ae)=\pi(a)\pi(e)=fe$ and hence $(ae)^n\in H_{fe}$ for some $n\in\IN$. Claim~\ref{cl:xhp} ensures that $a^nG=a^n(e^nG)=(ae)^nG=\{(ae)^n\}=\{a^ne\}$ is a singleton and hence $G\subseteq G_{a^n}$. Let $k\le n$ be the smallest number such that the subgroup $G\cap G_{a^k}$ has finite index in $G$. We claim that $k\ne 1$. Assuming that $G\cap G_a$ has finite index in $G$, we conclude that the quotient group $G/(G\cap G_a)$ is finite and hence bounded. Since the quotient group $H_e/G$ is bounded, the quotient group $H_e/(G\cap G_a)$ is bounded and so is the quotient group $H_e/G_a$. But this contradicts our assumption. This contradiction shows that $k\ne 1$. The minimality of $k$ ensures that the subgroup $G\cap G_{a^{k-1}}$ has infinite index in $G$. Since the group $G\cap G_{a^k}$ has finite index in $G$, the subgroup $G\cap G_{a^{k-1}}$ has infinite index in the group $G\cap G_{a^{k}}$. So, we can find an infinite set $I\subseteq G\cap G_{a^{k}}$ such that $x(G\cap G_{a^{k-1}})\cap y(G\cap G_{a^{k-1}})=\emptyset$ for any distinct elements $x,y\in I$. Observe that for any distinct elements $x,y\in I$ we have $a^kx=a^ke=a^ky$ and $a^{k-1}x\ne a^{k-1}y$ (assuming that $a^{k-1}x=a^{k-1}y$, we obtain that $a^{k-1}e=a^{k-1}xx^{-1}=a^{k-1}yx^{-1}$ and hence $yx^{-1}\in G\cap G_{a^{k-1}}$ which contradicts the choice of the set $I$).

Then the set $A=a^{k-1}I$ is infinite. We claim that $AA$ is a singleton. Indeed, for any $x,y\in I$ we have $a^{k-1}xa^{k-1}y=a^kxa^{k-2}y=a^kea^{k-2}y=a^{k-2}ea^ky=e^{k-2}ea^ke=a^{2k-2}e$. Therefore, $AA=\{a^{2k-2}e\}$. But the existence of such set $A$ contradicts our assumption.
\end{proof}

Let $\IQ_\infty=\{z\in\IC:\exists n\in\IN\;(z^n=1)\}$ be the
quasi-cyclic group, considered as a dense subgroup of the compact
Hausdorff group $\IT=\{z\in\IC:|z|=1\}$.

\begin{claim}\label{cl:inf} There exists a homomorphism $h:H_e\to
\IQ_\infty$ whose image $h[H_e]$ is infinite.
\end{claim}

\begin{proof}  By Lemma~\ref{l:period}, the group $H_e$ is
periodic, so for every $x\in H_e$ we can choose the smallest
number
$p(x)\in\IN$ such that $x^{p(x)}=e$. Since $H_e$ is unbounded,
there is a sequence $(x_n)_{n\in\IN}$ of elements such that
$p(x_n)>n\cdot\prod_{k<n}p(x_k)$ for all $n\in\IN$.

For every $n\in\IN$ let $G_n$ be the subgroup of $H_e$, generated
by the elements $x_1,\dots,x_n$. Let $G_0=\{e\}$ be the trivial
group and $h_0:G_0\to\{1\}\subset\IQ_\infty$ be the unique
homomorphism.  By induction, for every $n\in\IN$ we shall
construct
a homomorphism $h_n:G_n\to\IQ_\infty$ such that
$h_n{\restriction}G_{n-1}=h_{n-1}$ and $|h_n[G_n]|>n$. Assume that for some
$n\in\IN$ the homomorphism $h_{n-1}$ has been constructed. Consider the
cyclic subgroup $x_n^\IN$ generated by the
element $x_n$.
Consider the subgroup $Z=x_n^\IN\cap G_{n-1}\subseteq x_n^\IN$ and let
$\varphi:x_n^\IN\to\IQ_{\infty}$ be a homomorphism such that
$\varphi{\restriction}Z=h_{n-1}{\restriction}Z$ and $\varphi^{-1}(1)=(h_{n-1}{\restriction}Z)^{-1}(1)$.

Define the homomorphism $h_n:G_n\to\IQ_\infty$ by the formula
$h_n(cx)=\varphi(c)h_{n-1}(x)$ where $c\in x_n^\IN$ and $x\in
G_{n-1}$. To see that that $h_n$ is well-defined, take any
elements
$c,d\in x_n^\IN$ and $x,y\in G_{n-1}$ with $cx=dy$ and observe that
$d^{-1}c=yx^{-1}\in x_n^\IN\cap G_{n-1}$ and hence
$\varphi(d^{-1}c)=h_{n-1}(yx^{-1})$, which implies the desired
equality $\varphi(c)h_{n-1}(x)=\varphi(d)h_{n-1}(x)$. So, the
homomorphism $h_n$ is well-defined. It is clear that
$h_n{\restriction}G_{n-1}=h_{n-1}$ and the image $h_n[G_n]$ has cardinality
$$|h_n[G_n]|\ge
|h_n[x_n^\IN]|=|\varphi[x_n^\IN]|\ge|x_n^\IN/Z|=\frac{|x_n^\IN|}{|Z|}\ge
\frac{|x_n^\IN|}{|G_{n-1}|}\ge
\frac{p(x_n)}{\prod_{k<n}p(x_k)}>n.$$
After completing the inductive construction, consider the subgroup
$G=\bigcup_{n=1}^\infty G_n\subseteq H_e$ and the homomorphism
$h:G\to \IQ_\infty$ defined by $h{\restriction}G_n=h_n$ for all $n\in\IN$.

Taking into account that $|h[G]|\ge |h_n[G_n]|>n$ for all
$n\in\IN$, we conclude that the image $h[G]$ is infinite. By a
classical result of Baer \cite[21.1]{Fuchs}, the homomorphism $h$
can be extended to a homomorphism $\tilde h:H_e\to\IQ_\infty$. It
is clear that the image $\tilde h[H_e]$ is infinite.
\end{proof}

Denote by $\Phi$ the set of all homomorphisms from $H_e$ to
$\IQ_\infty$. By the classical Baer Theorem \cite[21.1]{Fuchs} on
extending homomorphisms into divisible groups, the homomorphisms
into $\IQ_\infty$ separate points of $H_e$, which implies that the
homomorphism $\vec\varphi:H_e\to\IQ_\infty^{\Phi}$,
$\vec\varphi:x\mapsto (\varphi(x))_{\varphi\in\Phi}$, is
injective.
Identify the group $H_e$ with its image $\vec\varphi[H_e]\subseteq
\IQ_\infty^\Phi$ in the compact topological group $\IT^{\Phi}$ and
let $\bar H_e$ be the closure of $H_e$ in $\IT^{\Phi}$.

By Claim~\ref{cl:inf}, the family $\Phi$ contains a homomorphism
$h:H_e\to\IQ_\infty$ with infinite image $h[H_e]$. The subgroup
$h[H_e]$, being infinite, is dense in $\IT$. The homomorphism $h$
admits a continuous extension $\bar h:\bar H_e\to\IT$, $\bar
h:(z_\varphi)_{\varphi\in\Phi}\mapsto z_h$.
The compactness of $\bar H_e$ and density of $h[H_e]=\bar h[
H_e]$ in $\IT$ imply that $\bar h[\bar H_e]=\IT$.

By Claim~\ref{cl:quot}, for every $a\in T$ the quotient group $H_e/G_a$ is bounded. So, we can find a
number $n_a\in\IN$ such that $x^{n_a}\in G_a$ for all $x\in H_e$.  Moreover, for any
non-empty finite set $F\subseteq T$ and the number
$n_F=\prod_{a\in
F}n_a\in\IN$, the intersection $G_F=\bigcap_{a\in F}G_a$ contains
the $n_F$-th power $x^{n_F}$ of any element $x\in H_e$.

Then for every $y\in h[H_e]\subseteq \IQ_\infty$, we get $y^{n_F}\in
h[G_{F}]$, which implies that the subgroup $h[G_{F}]$ is dense in
$\IT$. Let $\bar G_F$ be the closure of $G_F$ in the compact
topological group $\bar H_e$.  The density of the subgroup
$h[G_F]$
in $\IT$ implies that $\bar h[\bar G_F]=\overline{h[G_F]}=\IT$.

 By the compactness, $\bar h[\bigcap_{F\in [T]^{<\w}}\bar
G_{F}]=\bigcap_{F\in [T]^{<\w}}\bar h[\bar G_F]=\IT$. So, we can
fix an element $s\in \bigcap_{F\in [T]^{<\w}}\bar G_F\subseteq\bar
H_e$ whose image $\bar h(s)\in\IT$ has infinite order in the group
$\IT$. Then $s$ also has infinite order and its orbit
$s^\IN$ is disjoint with the periodic group $H_e$.

Consider the subsemigroup $S\subseteq\bar H_e$ generated by
$H_e\cup\{s\}$. Observe that $S\subseteq
\prod_{\varphi\in\Phi}\IQ_\varphi$ where $\IQ_\varphi$ is the
countable subgroup of $\IT$ generated by the set
$\IQ_\infty\cup\{\varphi(s)\}$.

It is clear that the subspace topology $\tilde\tau$ on $S$, inherited from the topological group
$\prod_{\varphi\in\Phi}\IQ_\varphi$ is Tychonoff and
zero-dimensional.
Then the topology $\tau'$ on $S$ generated by the base
$\{U\cap a\bar G_F:U\in\tilde\tau,\;a\in
\bar H_e,\;F\in[T]^{<\w}\}$ is also zero-dimensional.
It is easy to see that $(S,\tau')$ is a
topological semigroup and $s$  belongs to the closure of $H_e$ in
the topology $\tau'$. Finally, endow $S$ with the
topology
$\tau=\{U\cup D:U\in\tau',\;D\subseteq H_e\}$. The topology $\tau$
is
well-known in General Topology as the Michael modification of the
topology $\tau'$ (see \cite[5.1.22]{Eng}). Since the (group)
topology $\tau'$ is zero-dimensional, so is its Michael
modification $\tau$ (see \cite[5.1.22]{Eng}). Using the fact that
$S\setminus H_e$ is an ideal in $S$, it can be
shown that $(S,\tau)$ is a zero-dimensional  topological
semigroup, containing $H_e$ as a dense discrete subgroup. From now
on we consider $S$ as a topological semigroup, endowed
with the topology $\tau$.

Now consider the topological sum $Y=S\sqcup
(X\setminus H_e)$ of the topological space $S$ and the
discrete topological space $X\setminus H_e$. It is clear that
$Y$ contains $X$ as a proper dense discrete subspace.

It remains to extend the semigroup operation of $X$ to a
continuous
commutative semigroup operation on $Y$. In fact, for any
$a\in X$, $b\in H_e$ and $n\in\IN$ we should define the product
$a(bs^n)$. By the periodicity of the semigroup $X$, there is a
number $p\in\IN$ such that $f:=a^p$ is an idempotent. If $fe<e$,
then we put $a(bs^n)=ab$. If $fe=e$, then the element $ae$ has
power $(ae)^p=a^pe^p=fe=e$ and hence $ae$ belongs to the semigroup $\korin{\infty}{H_e}$. By Lemma~\ref{l:C-ideal}, the subgroup $H_e$ is an ideal in $\korin{\infty}{H_e}$. Consequently, $ae=aee\in H_e$. So, we can put
$a(bs^n)=(aeb)s^n$. The choice of
$x\in\bigcap_{F\in[T]^{<\w}}\bar G_F$ guarantees that the
extended
binary operation is continuous. Now the density of $X$ in $Y$ implies that the extended operation is commutative and
associative. Since $Y\in\mathsf{T_{\!z}S}$, the semigroup $X$ is not $\mathsf{T_{\!z}S}$-closed, which is a desired contradiction completing the proof of Lemma~\ref{l:subgrp2}.
\end{proof}

\begin{lemma}\label{l:subgrp} If a $\mathsf{T_{\!z}S}$-closed periodic semigroup $X$ has $X\cdot H(X)\subseteq Z(X)$, then each subgroup of $X$ is bounded.
\end{lemma}

\begin{proof} Assuming that $X$ contains an unbounded subgroup, we can apply Lemma~\ref{l:subgrp2} and find elements $e\in E(X)$, $x\in X$, and an infinite subset $A\subseteq x\cdot H_e$ such that the set $AA$ is a singleton. Since $A\subseteq X\cdot H(X)\subseteq Z(X)$, we can apply Lemma~\ref{l:AA} and conclude that the semigroup $X$ is not $\mathsf{T_{\!z}S}$-closed. But this contradicts our assumption.
\end{proof}

\begin{lemma}\label{l:power-periodic} Let $X$ be a $\mathsf{T_{\!z}S}$-closed semigroup and $e\in E(X)$ be an idempotent such that the semigroup $H_e\cap Z(X)$ is bounded. Then for any sequence  $(x_n)_{n\in\w}$ in $(\korin{\infty}{\,e}\cap Z(X))\setminus H_e$ there exists $n\in\w$ such that $x_n\notin \{x_{n+1}^p:p\ge 2\}$.
\end{lemma}

\begin{proof}
To derive a contradiction assume that there exists a sequence $(x_n)_{n\in\w}$ in $(\korin{\infty}{\,e}\cap Z(X))\setminus H_e$ such that for every $n\in\IN$ there exists $p_n\ge 2$ such that $x_{n-1}=x_{n}^{p_n}$. Since the semigroup $H_e\cap Z(X)$ is bounded, there exists $n_e\in\mathbb{N}$ such that $z^{n_e}=e$ for every $z\in H_e\cap Z(X)$.

Consider the additive subsemigroup $Q_+=\big\{\frac{k}{p_1\cdots
p_n}:k,n\in\IN\big\}$ of the semigroup of positive rational
numbers
endowed with the binary operation of addition of rational numbers.
Let $h:Q_+\to \korin{\infty}{\,e}\,\cap Z(X)$ be the (unique) homomorphism such that
$h(\frac{1}{p_1\cdots p_n})=x_n$ for all $n\in\IN$. Then $h(1)=h(\frac{p_1}{p_1})=x_1^{p_1}=x_0\notin H_e$. By Lemma~\ref{l:period}, the center $Z(X)$ of the $\mathsf{T_{\!z}S}$-closed semigroup $X$ is periodic and hence $H_e\cap Z(X)$ is a periodic subgroup of $\korin{\infty}{\,e}\cap Z(X)$. By
Lemma~\ref{l:C-ideal}, the subgroup $H_e\cap Z(X)$ is an ideal in
$\korin{\infty}{\,e}\,\cap Z(X)$. Consequently, the preimage $h^{-1}[H_e]=h^{-1}[H_e\cap Z(X)]$ is an upper set in
$Q_+$, which means that for any points $q<r$ in $Q_+$ with $q\in
h^{-1}[H_e]$ we get $r\in h^{-1}[H_e]$. Then
$L=h^{-1}[\korin{\infty}{\,e}\setminus
H_e]$ is a lower set, which  contains $1$ and hence contains the
interval $Q_+\cap(0,1]$. We claim that the restriction $h{\restriction}L$
injective. Assuming that $h(a)=h(b)$ for some distinct points
$a<b$
in $L$, we can find natural numbers $k$ and $n<m$ such that
$a=\frac{n}{p_1\cdots p_k}$ and $b=\frac{m}{p_1\dots p_k}$. Then
$x^{n}_k=h(a)=h(b)=x^{m}_k$. By Theorem 1.9 from~\cite{Clifford-Preston-1961}, $x^n_k\in H_e$ and, therefore,
$a=\frac{n}{p_1\cdots p_k}\in h^{-1}[H_e]$. But this contradicts
the choice of $a\in L\subseteq Q_+\setminus h^{-1}[H_e]$.

Let $s=\sup L\in(0,+\infty)$ and $W=\{q\in
n_eQ_+:\frac{s}2<q<s\}\subseteq L$. The injectivity of $h{\restriction}L$
guarantees that the set $h[W]$ is infinite. Observe that for every
points $a,b\in W$ we get $a+b>2\frac{s}2=s$ and hence $h(a+b)\in
H_e\cap Z(X)$ and thus $h(a+b)=h(a+b)e$.  Find $z\in Q_+$ such that $a+b=n_ez$. Then $h(a+b)=h(n_ez)e=h(z)^{n_e}e=(h(z)e)^{n_e}=e$ by the choice of $n_e$ and the inclusion $h(z)e\in \korin{\infty}{H_e}\cdot H_e\subseteq H_e$ (see Lemma~\ref{l:C-ideal}). This implies that the infinite set $A=h[W]\subseteq Z(X)$ has $AA=\{e\}$. Applying Lemma~\ref{l:AA}, we conclude that the semigroup $X$ is not $\mathsf{T_{\!z}S}$-closed which contradicts our assumption.
\end{proof}

\section{Proof of Theorem~\ref{t:main}}\label{s:main}

We should prove that a commutative semigroup $X$ is $\C$-closed if and only if $X$ is periodic, chain-finite, all subgroups of $S$ are bounded and for every infinite set $A\subseteq X$ the set $AA$ is not a singleton.

The ``only if'' part follows from   Lemmas~\ref{l:chain}, \ref{l:AA}, \ref{l:period} and \ref{l:subgrp}. To prove the ``if'' part, assume that $X$ is periodic, chain-finite, all subgroups of $S$ are bounded and for every infinite set $A\subseteq X$ the set $AA$ is not a singleton. By the periodicity, $X=\bigcup_{e\in E(X)}\korin{\infty}{\,e}$. Consider the map  $\pi:X\to E(X)$ assigning to each $x\in X$ the unique idempotent in the monogenic semigroup $x^\IN$. By Proposition~\ref{p:pi-homo}, the map $\pi$ is a semigroup homomorphism. By Lemma~\ref{l:single}, for every idempotent $e\in E(X)$ the semigroup $\korin{\infty}{\,e}$ is $\C$-closed. Since $X$ is chain-finite, so is the semilattice $E(X)$. Applying Lemma~\ref{l:cool}, we conclude that the semigroup $X$ is $\C$-closed.

\section{$\C$-closedness of quotient semigroups}\label{s:quotient}

In this section we prove some lemmas that will be used in the proof of Theorem~\ref{t:mainP}.

\begin{lemma}\label{l:quotient} Let $X$ be a periodic semigroup, $e\in E(X)$ and $Z_n:=\{z\in  Z(X):z^n\in H_e\}$ for $n\in\IN$. If for some $\ell\in\IN$ the set $Z_\ell\setminus H_e$ is infinite, then there exist a finite set $F\subseteq Z_\ell$ and an infinite set $A\subseteq Z_\ell\setminus FX^1$ such that $AA\subseteq F\cup H_e\subseteq FX^1$.
\end{lemma}

\begin{proof} Lemma~\ref{l:C-ideal} implies that $Z_n\subseteq Z_{n+1}$ for all $n\in\w$. Let $Z_\infty=\bigcup_{n\in\IN}Z_n=Z(X)\cap\korin{\infty}{H_e}$. By our assumption, there exists a number $\ell\in\IN$ such that the set $Z_\ell\setminus H_e$ is infinite. We can assume that $\ell$ is the smallest number with this property. The obvious equality $Z_1= Z(X)\cap H_e$ implies that $\ell\ge 2$ and hence the set $Z_{\ell-1}\setminus H_e$ is finite by the minimality of $\ell$.

Choose any sequence $(z_n)_{n\in\w}$ of pairwise distinct elements of the infinite set $Z_\ell\setminus  Z_{\ell-1}$.

\begin{claim}\label{cl:6.2} For every $z\in Z_\ell$ we have $z^2\in Z_{\ell-1}$.
\end{claim}

\begin{proof} In the case $\ell=2$, by the definition of $\ell$, we get that $z^2\in Z(X)\cap H_e=Z_1$. Assuming $\ell> 2$, we have that $$(z^2)^{\ell-1}=z^{2\ell-2}=z^\ell z^{\ell-2}\in H_e\cdot \korin{\infty}{H_e}\subseteq H_e,$$as $H_e$ is an ideal in $\korin{\infty}{H_e}$ by Lemma~\ref{l:C-ideal}. Then $z^2\in Z_{\ell-1}$ according to the definition of $Z_{\ell-1}$.
\end{proof}

\begin{claim}\label{cl:6.3} For every $n\in\IN$ we have $Z_\infty\cap Z_nX^1\subseteq Z_n$.
\end{claim}

\begin{proof} For any $z\in Z_n$ and $x\in X^1$, with $zx\in Z_\infty$, we should prove that $(zx)^n\in H_e$. The inclusion $z\in Z_n\subseteq Z(X)$ implies $z^n\in H_e\cap Z(X)$ and $\pi(z)=e\in Z(X)$.  By Claim~\ref{cl:pi-Z}, $e=\pi(zx)=\pi(z)\pi(x)=e\pi(x)$ and $\pi(ex^n)=\pi(e)\pi(x)=e$ and thus $ex^n\in\korin{\infty}{H_e}$. Then $$(zx)^n=z^nx^n=(z^ne)x^n=z^n(ex^n)\in H_e \cdot\korin{\infty}{H_e}\subseteq H_e.$$
\end{proof}

If for some $z\in Z_\ell$, the set $A=(Z_\ell\setminus Z_{\ell-1})\cap (zX^1)$ is infinite, then for the finite set $F=(Z_{\ell-1}\setminus H_e)\cup\{e\}$ we have
$$A\cap FX^1\subseteq (Z_\ell\setminus Z_{\ell-1})\cap Z_{\ell-1}X^1=\emptyset$$ by Claim~\ref{cl:6.3}. On the other hand, applying Claims~\ref{cl:6.2} and \ref{cl:6.3}, we obtain
$$AA\subseteq Z_{\ell}\cap z^2X^1\subseteq Z_\ell\cap Z_{\ell-1}X^1\subseteq Z_{\ell-1}\subseteq F\cup H_e\subseteq FX^1.$$ Therefore, the finite set $F$ and the infinite set $A$ have the properties required in Lemma~\ref{l:quotient}.

So, we assume that for every $z\in Z_\ell$, the set $(Z_\ell\setminus Z_{\ell-1})\cap (zX^1)$ is finite.

Let $T=\{\langle i,j,k\rangle\in \w\times\w\times\w:i<j<k\}$ and $\chi:T\to \{0,1,2\}$ be the function defined by the formula
$$
\chi(i,j,k)=\begin{cases}
0&\mbox{if $z_iz_j\in Z_{\ell-1}$};\\
1&\mbox{if $z_iz_j\notin Z_{\ell-1}$ and $z_iz_j\ne z_iz_k$};\\
2&\mbox{if $z_iz_j\notin Z_{\ell-1}$ and $z_iz_j=z_iz_k$}.
\end{cases}
$$
By the Ramsey Theorem 5 in \cite{Ramsey}, there exists an infinite set $\Omega\subseteq\w$ such that $\chi[T\cap\Omega^3]=\{c\}$ for some $c\in\{0,1,2\}$.

If $c=0$, then the infinite set $A=\{z_n:n\in\Omega\}$ has $AA\subseteq Z_{\ell-1}\subseteq F\cup H_e\subseteq FX^1$ for the finite set $F=(Z_{\ell-1}\setminus H_e)\cup\{e\}$. On the other hand, $$A\cap FX^1\subseteq (Z_\ell\setminus Z_{\ell-1})\cap (Z_{\ell-1} X^1)=\emptyset$$by Claim~\ref{cl:6.3}.

If $c=1$, then for any $i\in \Omega$ the set $\{z_iz_j:i<j\in\Omega\}$ is an infinite subset of the set $z_i Z_\ell\setminus Z_{\ell-1}\subseteq (Z_\ell\setminus Z_{\ell-1})\cap (z_iX^1)$, which is finite by our assumption. Therefore, the case $c=1$ is impossible.

If $c=2$, then $z_iz_j=z_iz_k\notin Z_{\ell-1}$ for any numbers $i<j<k$ in $\Omega$. Then for every $i<k$ in $\Omega$ we have $z_iz_k=z_iz_{i^+}$ where $i^+=\min\{j\in \Omega:i<j\}$.
\smallskip

Now consider two cases.
\smallskip


1) The set $A=\{z_iz_{i^+}:i\in \Omega\}\subseteq Z_\ell\setminus Z_{\ell-1}$ is infinite.
Observe that for any numbers $i<j$ in $\Omega$, applying Claims~\ref{cl:6.2} and \ref{cl:6.3}, we obtain
$$z_iz_{i^+}z_jz_{j^+}=z_iz_jz_jz_{j^+}\in z_j^2 Z_\ell\subseteq Z_{\ell-1}Z_\ell\subseteq
Z_{\ell-1}$$ and $z_iz_{i^+}z_iz_{i^+}\in z_i^2Z_\ell\subseteq Z_{\ell-1}Z_\ell\subseteq Z_{\ell-1}$. Then $AA\subseteq Z_{\ell-1}\subseteq F\cup H_e$ for the finite set $F=(Z_{\ell-1}\setminus H_e)\cup\{e\}$. Also $A\cap FX^1\subseteq (Z_\ell\setminus Z_{\ell-1})\cap Z_{\ell-1}X^1=\emptyset$ by Claim~\ref{cl:6.3}.
\smallskip

2) The set $C=\{z_iz_{i^+}:i\in\Omega\}$ is finite. In this case there exists an element $c\in C$ such that the set $\Lambda=\{i\in\Omega:c=z_iz_{i^+}\}$ is infinite. By our assumption, the set $(Z_\ell\setminus Z_{\ell-1})\cap cX^1$ is finite. Then, taking into account Claim~\ref{cl:6.2}, the infinite set $A=\{z_i:i\in\Lambda\}\setminus cX^1$ satisfies
$$AA=\{z_i^2:i\in\Lambda\}\cup\{z_iz_j:i,j\in\Lambda,\;i<j\}\subseteq Z_{\ell-1}\cup \{z_iz_{i^+}:i\in\Lambda\}=Z_{\ell-1}\cup\{c\}\subseteq F\cup H_e\subseteq FX^1$$for the finite set $F=(Z_{\ell-1}\setminus H_e)\cup\{c,e\}$. Also $$A\cap FX^1\subseteq ((Z_\ell\setminus Z_{\ell-1})\cap (Z_{\ell-1}X^1))\cup (A\cap cX^1)=\emptyset.$$

In all cases we have constructed an infinite set $A\subseteq Z_\ell$ and a finite set $F\subseteq Z_\ell$ such that $AA\subseteq F\cup H_e\subseteq FX^1$ and $A\cap FX^1=\emptyset$.
\end{proof}

\begin{remark}\label{lapochka}
The finite set $F$ from the previous lemma can be forced to have at most two elements. For this we need to use one more time the Ramsey Theorem. Let $A$ be the infinite set constructed in the previous lemma. Recall that $AA\subseteq F\cup H_e$ and $A\cap FX^1=\emptyset$. Write the set $F\cup\{e\}$ as $\{f_0,f_1,\dots,f_n\}$ where $f_0=e$.  The Pigeonhole Principle implies that there exists $k\leq n$ and an infinite subset $B\subseteq A$ such that $\{x^2:x\in B\}\subseteq H_e$ if $k=0$ and $\{x^2:x\in B\}=\{f_k\}$ if $k>0$. By $[B]^2$ we denote the set of all two-element subsets of $B$. Consider the function $\chi: [B]^2\rightarrow \{0,\ldots,n\}$ defined by the formula:
$$
\chi(\{a,b\})=
\begin{cases}
0&\mbox{if $ab\in H_e$};\\
i&\mbox{if $ab=f_i$ for some $i\in\{1,\dots,n\}$.}
\end{cases}
$$
By the Ramsey Theorem, there exist a number $i\in\{0,\dots,n\}$ and an infinite subset $A'\subseteq B$ such that $\chi(\{x,y\})=i$ for any distinct elements $x,y\in A'$. Then for the set $F'=\{f_k,f_i\}$ we have $A'A'\subseteq F'\cup H_e$ and $A'\cap F'X^1\subseteq A\cap FX^1=\emptyset$.
\end{remark}

\begin{lemma}\label{l:powers} Let $X$ be an ideally $\mathsf{T_{\!z}S}$-closed semigroup such that for some $e\in E(X)\cap Z(X)$ the semigroup $H_e\cap Z(X)$ is bounded. Then the set $(\korin{\infty}{H_e}\cap Z(X))\setminus H_e$ is finite.
\end{lemma}

\begin{proof} To derive a contradiction, assume that the set $(\korin{\infty}{H_e}\cap Z(X))\setminus H_e$ is infinite.
For every $n\in\IN$ consider the set
$$Z_n=\{z\in Z(X):z^n\in H_e\}$$
and let $Z_\infty=\bigcup_{n\in\IN}Z_n=Z(X)\cap\korin{\infty}{H_e}$.

If for some $\ell\in\IN$ the set $Z_\ell\setminus H_e$ is infinite, then we can apply Lemma~\ref{l:quotient} and find a finite set $F\subset Z_\ell\subseteq Z(X)$ and infinite set $A\subseteq Z_\ell$ such that $AA\subseteq F\cup H_e\subseteq FX^1$ and $A\cap (FX^1)=\emptyset$. Consider the ideal $I=FX^1$ and the quotient semigroup $X/I$. Then the quotient image $q[A]$ of $A$ in $X/I$ is an infinite set in $Z(X/I)$ such that $q[A]q[A]=q[I]$ is a singleton. By Lemma~\ref{l:AA}, the semigroup $X/I$ is not $\mathsf{T_{\!z}S}$-closed, which contradicts our assumption. This contradiction shows that for every $n\in\IN$ the set $Z_n\setminus H_e$ is finite.

Since the semigroup $Z(X)\cap H_e$ is bounded, there exists $p\in\IN$ such that $x^p=e$ for all $x\in Z(X)\cap H_e$. Consider the subsemigroup $P=\{z^p:z\in Z_\infty\}$ in $Z_\infty$.

\begin{claim} $P\cap H_e=\{e\}$.
\end{claim}

\begin{proof}
 Given any element $x\in P\cap H_e$, find $z\in Z_\infty$ such that $x=z^p$. Lemma~\ref{l:C-ideal} implies that $ze\in \korin{\infty}{\,e}\cdot e\subseteq H_e$ and hence  $(ze)^p=e$ by the choice of $p$. Then $x=xe=z^pe=(ze)^p=e$.
\end{proof}

\begin{claim} For every $n\in\w$ the set $P\setminus Z_n$ is not empty.
\end{claim}

\begin{proof} Assuming that $P\setminus Z_n=\emptyset$, we conclude that $P\subseteq Z_n$ and hence $Z_\infty\subseteq Z_{pn}$. Then the set $Z_\infty\setminus H_e\subseteq Z_{pn}\setminus H_e$ is finite, which contradicts our assumption.
\end{proof}

Consider the tree $$T=\bigcup_{n\in\w}\big\{(t_k)_{k\in n}\in P^n:t_0=e\;\wedge\;\big(\forall k\in n\setminus\{0\}\;\;(t(k)\in Z_{2^k}\setminus H_e\;\wedge\;t(k)^2=t(k-1))\big)\big\}$$endowed with the partial order of inclusion of functions. Since the sets $Z_n\setminus H_e$ are finite, the tree $T$ has finitely many branching points at every vertex.  On the other hand, this tree has infinite height. This follows from the fact that for every element $z\in P\setminus Z_{2^k}$, there exists $n>k$ such that $z^{2^n}\in P\cap H_e=\{e\}$ but $z^{2^{n-1}}\notin H_e$. Then the sequence $(z^{2^{n-i}})_{i\in k}$ belongs to the tree $T$. By K\H onig's Lemma 5.7~\cite{Kun}, the tree $T$ has an infinite branch which is a sequence $(z_n)_{n\in\w}$ in $P$ such that  $z_0=e$ and $z_n^2=z_{n-1}$, $z_n\in Z_{2^n}\setminus H_e$, for all $n\in\IN$.
But the existence of such a sequence contradicts Lemma~\ref{l:power-periodic}.
\end{proof}

\section{Proof of Theorem~\ref{t:mainP}}\label{s:mainP}

Given a commutative semigroup $X$, we should prove the equivalence of the following conditions:
\begin{enumerate}
\item $X$ is projectively $\C$-closed;
\item $X$ is ideally $\C$-closed;
\item $X$ is chain-finite, almost Clifford, and all subgroups are bounded.
\end{enumerate}

The implication $(1)\Ra(2)$ is trivial.
\smallskip

$(2)\Ra(3)$ Assume that $X$ is ideally $\C$-closed. Then $X$ is $\C$-closed and by Lemmas~\ref{l:chain}, \ref{l:period}, \ref{l:subgrp}, $X=Z(X)$ is chain-finite, periodic, and all subgroups of $X$ are bounded. It remains to prove that $X$ is almost Clifford. By Lemma~\ref{l:powers}, for every $e\in E(X)$ the set $\korin{\infty}{H_e}\setminus H_e$ is finite. Assuming that the semigroup $X$ is not almost Clifford, we conclude that the set $B=\{e\in E(X):\korin{\infty}{H_e}\ne H_e\}$ is infinite. Since the semilattice $E(X)$ is chain-finite, we can apply the Ramsey Theorem and find an infinite antichain $C\subseteq B$ (the latter means that $xy\notin\{x,y\}$ for any distinct elements $x,y\in C$). Let $R=\{e\in E(X): \exists c\in C\;\;(e<c)\}$. 
 It is straightforward to check that $R$ is an ideal in $E(X)$. Lemma~\ref{l:C-ideal} and Proposition~\ref{p:pi-homo} imply that $J=\bigcup_{e\in R}\korin{\infty}{H_e}$ is an ideal in $X$.
Let us show that the set $I=J\cup\bigcup_{e\in C}H_e$ is an ideal in $X$. Indeed, fix any $x\in X$ and $y\in I$. If $y\in J$, then $xy\in J\subseteq I$, because $J$ is an ideal in $X$. Assume that $y\in H_e$ for some $e\in C$. Then Proposition~\ref{p:pi-homo} implies that either $xy\in \korin{\infty}{H_f}\subseteq I$ for some idempotent $f<e$, or $xy\in \korin{\infty}{H_e}$. In the latter case, since $H_e$ is an ideal in $\korin{\infty}{H_e}$, we get that $xy=xye\in H_e\subseteq I$. So, $I$ is an ideal in $X$.
Since $X$ is ideally $\C$-closed, the quotient semigroup $X/I$ is $\C$-closed. By the choice of the set $C$, for every $e\in C$ the set $\korin{\infty}{H_e}\setminus H_e$ is not empty. Let us show that the set $\korin{\infty}{H_e}\setminus H_e$ contains an element $a_e$ such that $a_e^2\in H_e$. Pick any $b\in \korin{\infty}{H_e}\setminus H_e$. If $b^2\in H_e$, then put $a_e=b$. Otherwise, let $n$ be the smallest integer such that $b^n\in H_e$, which exists by the periodicity of $X$. Note that $n>2$ and $b^{n-1}\in  \korin{\infty}{H_e}\setminus H_e$. Put $a_e=b^{n-1}$ and observe that $$a_e^2=(b^{n-1})^2=b^{2n-2}=b^nb^{n-2}\in H_eb^{n-2}\subseteq H_e.$$ We claim that for the set $A=\{a_e:e\in C\}\subseteq X\setminus I$ we have $AA\subseteq I$.
Since the set $C\subseteq E(X)$ is an antichain, for every distinct $e,f\in C$ we have $ef<e$ and hence $ef\in R\subset I$. Then Proposition~\ref{p:pi-homo} implies that for every distinct $e,f\in C$, $a_e a_f\in \korin{\infty}{H_{e}}\korin{\infty}{H_{f}}\subseteq \korin{\infty}{H_{ef}}\subseteq I$. Hence $AA\subseteq I$.
Moreover,  the image $q[A]$ of $A$ in the quotient semigroup $X/I$ is an infinite subset of $X/I$ such that $q[A]q[A]=q(I)$ is a singleton, which contradicts Lemma~\ref{l:AA}.
\smallskip

$(3)\Ra(1)$ Assume that $X$ is chain-finite, almost Clifford and all subgroups of $X$ are bounded. Let us show that $X$ is periodic. Elements of $H(X)$ are periodic, since maximal subgroups are bounded. If $x\in X\setminus H(X)$, then the finiteness of $X\setminus H(X)$ and the fact that $E(X)\subseteq H(X)$ imply that for some $n\in\mathbb N$ the element $x^n\in H(X)$. Consequently $x^{nm}=e\in E(X)$ for some large enough $m$, because maximal subgroups are bounded.
By Theorem~\ref{t:main}, the projective $\C$-closedness of $X$ will be proved as soon as we check that for any congruence $\approx$ on $S$ the quotient semigroup $Y=X/_\approx$ is periodic, chain-finite, all subgroups of $X/_\approx$ are bounded and for any infinite subset $A\subseteq X/_\approx$ the set $AA$ is not a singleton.

Let $q:X\to Y=X/_\approx$ be the quotient homomorphism. The periodicity of $X$ implies the periodicity of $Y$. To see that $Y$ is chain-finite, observe that for every $e\in E(Y)$ the (periodic) semigroup $q^{-1}(e)$ contains an idempotent. This implies that $E(Y)=q[E(X)]$. Since $X$ is chain-finite, its maximal semilattice $E(X)$ is chain-finite. By Corollary~\ref{BBT1}, $E(X)$ is projectively $\C$-closed and then so is its homomorphic image $E(Y)$. Using Corollary~\ref{BBT1} one more time, we obtain that the semilattice $E(Y)$ is chain-finite. The following claim implies that the semigroup $Y$ is almost Clifford and all subgroups of $Y$ are bounded.

\begin{claim} For any idempotent $e\in E(Y)$ there exists an idempotent $s\in E(X)$ such that $q[H_s]=H_e$.
\end{claim}

\begin{proof} Since $X$ is chain-finite, the semilattice $E(X)\cap q^{-1}(e)$ contains the smallest idempotent $s$. We claim that $H_e=q[H_s]$. In fact, the inclusion $q[H_s]\subseteq H_e$ is trivial. To see that $H_e\subseteq q[H_s]$, take any element $y\in H_e\subseteq Y$ and find $x\in X$ with $q(x)=y$. Find $n\in\IN$ such that $x^n\in E(X)$ and $y^n=e$. It follows from $y=q(x)$ that $e=y^n=q(x^n)$ and hence $s=sx^n$ by the minimality of $s$. By Proposition~\ref{p:pi-homo}, $\pi(sx)=\pi(s)\pi(x)=sx^n=s$ and then
$$sx=(ss)x=s(sx)=\pi(sx)sx\in H_{\pi(sx)}=H_s$$by Lemma~\ref{l:C-ideal}. Finally, $y=ey=q(s)q(x)=q(sx)\in q[H_s]$ and hence $H_e=q[H_s]$.
\end{proof}

It remains to prove that for any infinite subset $A\subseteq Y$ the set $AA$ is not a singleton. This follows from the next lemma.

\begin{lemma} For any infinite set $A$ in an almost Clifford semigroup $S$ the set $AA$ is infinite.
\end{lemma}

\begin{proof} Since $S$ is almost Clifford, the set $A\cap H(S)$ is infinite. If for some idempotent $e\in E(S)$ the intersection $A\cap H_e$ is infinite, then for any $a\in A\cap H_e$ the set $a(A\cap H_e)\subseteq (AA)\cap H_e\subseteq AA$ is infinite (since shifts in groups are injective) and hence $AA$ is not a singleton. So,  assume that for every $e\in E(S)$ the intersection $A\cap H_e$ is finite.

Then the set $E=\{e\in E(S):A\cap H_e\ne\emptyset\}$ is infinite. For every $e\in E$, fix an element $a_e\in A\cap H_e$ and observe that $a_e^2\in AA\cap H_e$, which implies that the set $AA\supseteq\{a_e^2:e\in E\}$ is infinite.
\end{proof}

\section{Some open problems}\label{s:final}

In this section we ask two open problems, motivated by the results, obtained in this paper.

\begin{question}
Does there exist a $\Zero$-closed semigroup which is not $\mathsf{T_{\!2}S}$-closed?
\end{question}




\begin{question}\label{id}
Is every ideally $\C$-closed (semi)group projectively $\C$-closed?
\end{question}




\end{document}